\documentclass[12pt]{article}

\usepackage{amsthm,amsmath,amsfonts,amssymb,yfonts}
\usepackage[english]{babel}
\usepackage{graphicx}
\usepackage{verbatim}
\usepackage{bm}
\usepackage{color}
\usepackage{float}
 \usepackage{booktabs}
 \usepackage[table,xcdraw]{xcolor}

\usepackage{enumitem}

\usepackage{tikz,lipsum,lmodern}
\usepackage[most]{tcolorbox}
\usepackage{multicol}
\usepackage{tabularx}
\usepackage[utf8]{inputenc}
\usepackage{array}
\usepackage{wrapfig}
\usepackage{multirow}

\usepackage[utf8]{inputenc}
\usepackage[T1]{fontenc}
\usepackage{authblk}

\numberwithin{equation}{subsection}

\let\oldsection\section
\renewcommand{\section}{
	\renewcommand{\theequation}{\thesection.\arabic{equation}}
	\oldsection}
\let\oldsubsection\subsection
\renewcommand{\subsection}{
	\renewcommand{\theequation}{\thesubsection.\arabic{equation}}
	\oldsubsection}

\newtheorem{theorem}{Theorem}[section]
\newtheorem{proposition}[theorem]{Proposition}
\newtheorem{lemma}[theorem]{Lemma}

\newtheorem{definition}[theorem]{Definition}

\newtheorem{remark}[theorem]{Remark}

\title{Construction of directed strongly regular graphs  with nontrivial automorphisms}

\author{
	{Marija Maksimović}\\
	{\small Faculty of Mathematics}\\
	{\small University of Rijeka}\\
	{\small Radmile Matejčić 2, 51000 Rijeka, Croatia }\\
	{\small mmaksimovic@math.uniri.hr} \and 
	{Sanja Rukavina}\\
	{\small Faculty of Mathematics}\\
	{\small University of Rijeka}\\
	{\small Radmile Matejčić 2, 51000 Rijeka, Croatia }\\
	{\small sanjar@math.uniri.hr}
}
\begin{document}
\maketitle
\begin{abstract}
In this paper we present a method for constructing directed strongly regular graphs with assumed action of an automorphism group. The application of this method leads to first examples of directed strongly regular graphs with parameters $(22,9,6,3,4)$. We have shown that an automorphism of prime order acting on such a graph can only be of order two or three. Furthermore, we have constructed $472$ directed strongly regular graphs with parameters $(22,9,6,3,4)$ and classified all these graphs with an automorphism of order three.
\end{abstract}



\section{Introduction} 

We assume familiarity with the basic facts and concepts of graph theory. For further reading on the topic, we refer the reader to \cite{BCN, bro2, duval}.

In \cite{duval} Duval introduced the directed graph version of a strongly regular graph. A {\it directed graph} is a pair $\Gamma=(V,E)$ consisting of a set $V$, referred to as the vertex set of $\Gamma$
and a set $E \subseteq V \times V$, which is called the edge set of $\Gamma$. A {\it directed strongly regular graph with parameters} $(v,k,t, \lambda, \mu)$ is a directed graph $\Gamma=(V,E)$ on $v$ vertices without loops such that
 each vertex $x \in V$ has in-degree and out-degree $k$,
	 each vertex $x \in V$ has $t$ out-neighbors that are also in-neighbors of $x$,
 and for each pair of vertices $x, y \in V$, the number of directed paths of length two from $x$ to $y$ is $\lambda$ if $(x,y)\in E$, and $\mu$ if $(x,y)\notin E$.
We call such a graph $\Gamma$ a $DSRG(v,k,t,\lambda,\mu).$
 Its adjacency matrix $A= [a_{ij}]$ is a $(0,1)$-matrix with rows and columns indexed by $V$ and defined by $a_{xy}=1$ if and only if $(x,y)\in E$.
The adjacency matrix $A$ of $\Gamma$ satisfies
\begin{gather}
	A^2+(\mu-\lambda)A-(t-\mu)I=\mu J \label{1}\\
	AJ=JA=kJ \label{2}
\end{gather}
where $I$ denotes the identity matrix and $J$ the all-1 matrix of order $v$.

Since the introduction of directed strongly regular graphs by Duval \cite{duval}, this concept has attracted the attention of many researchers. An overview of the known results and parameters for which such graphs have been constructed is presented in \cite{bro1}.
According to \cite{bro1}, the smallest open case of admissible parameters for which the existence of a directed strongly regular graph has not been proved is $(22,9,6,3,4)$. We have not found an example of such a graph in the literature either.

In this paper we give a method for the construction of directed strongly regular graphs with presumed action of an automorphism group by adapting the method used for the construction of other types of incidence structures, including strongly regular graphs (see \cite{beh,cm}). The method is based on a construction of orbit matrices, which are then refined to obtain the adjacency matrix of a directed strongly regular graph.

 The work is structured as follows. After describing the basic properties of orbit matrices of directed strongly regular graphs and introducing the notation  in Section \ref{om}, a method for constructing directed strongly regular graphs from their orbit matrices is presented in Section \ref{con}. By applying the developed method to a construction of directed strongly regular graphs with parameters $(22,9,6,3,4)$, we constructed $472$ nonisomorphic graphs in Section \ref{22}. These are the first examples of directed strongly regular graphs with parameters $(22,9,6,3,4)$. Moreover, we show that an automorphism of prime order acting on such a graph can only be of order two or three, and classify all DSRGs$(22,9,6,3,4)$ with an automorphism of order three. Our results are summarized in Section \ref{concl}.

The orbit matrices and directed strongly regular graphs constructed in this paper were constructed and analyzed using GAP \cite{GAP}.

\section{Orbit matrices of directed strongly regular graphs} \label{om}

Orbit matrices of strongly regular graphs have been considered in \cite{beh,cm}. The definition and some properties of orbit matrices of directed strongly regular graphs can be found in \cite{csz}, where orbit matrices of known directed strongly regular graphs were used to construct new directed strongly regular graphs with prescribed automorphism group using a genetic algorithm. Our goal is to adapt the method used for the construction of other types of incidence structures (see, for example, \cite{beh,cm,c-r-metrika,janko}) and to construct matrices corresponding to the definition of orbit matrices of directed strongly regular graphs, which are then used for the construction of all nonisomorphic directed strongly regular graphs resulting from the obtained matrices. Thus, in this section we give an overview of the basic facts and properties of orbit matrices that we need for our construction and introduce the notation for the rest of the paper.

Let $\Gamma=(V,E)$ be a \text{DSRG}$(v, k, t, \lambda, \mu)$, and let $A$ be its adjacency matrix. 
An \emph{automorphism} of $\Gamma$ is a permutation $\sigma$ of the vertex set $V$ such that, for every pair of vertices $x,y \in V$, $(x,y) \in E$  if and only if $(\sigma(x),\sigma(y)) \in E$.
The set of all automorphisms of $\Gamma$ forms a group under composition called the \textit{full automorphism group} of $\Gamma$ and denoted by $Aut(\Gamma)$. 
Suppose an automorphism group $G$ of $\Gamma$ partitions the vertex set $V$ into $b$ orbits $O_1, O_2, \dots, O_b$ with corresponding sizes $n_1, n_2, \dots, n_b$, respectively. The adjacency matrix $A$ can be subdivided into submatrices $[A_{ij}]$, where   $A_{ij}$ represents the submatrix whose rows and columns are indexed by  elements of $O_i$ and $O_j$, respectively.

Define two matrices $C = [c_{ij}]$ and $R = [r_{ij}]$ for $1 \leq i, j \leq b$, such that:
\[
c_{ij} = \text{column sum of } A_{ij}, \quad r_{ij} = \text{row sum of } A_{ij}.
\]
$C$ is the \emph{column orbit matrix} of $\Gamma$ with respect to the group $G$, and $R$ is the \emph{row orbit matrix} of $\Gamma$ with respect to the group $G$. The matrices $C$ and $R$ are connected by the following equation:
\begin{align}
r_{ij} n_i = c_{ij} n_j \label{3}.
\end{align}
Let $W = A^2$. Using the same partitioning of the vertex set $V$ into orbits, the matrix $W$ can be partitioned as $W = [W_{ij}]$, where $i$ and $j$ are the indices of the orbits.

Finally, let $N = \text{diag}(n_1, n_2, \dots, n_b)$ and define a $b \times b$ matrix $S = [s_{ij}]$, where $s_{ij}$ is equal to the sum of all entries in $W_{ij}$. The following result will be useful.

\begin{proposition}\label{cnr} With reference to the above notations, the following holds:
	$$CNR=S.$$
	\end{proposition}

\begin{proof}
	Let $\alpha_i$ be a vector of size $v$ such that
	$$\alpha_i(j)=\begin{cases}
		1, \quad j\in O_i\\
		0, \quad j\notin O_i
	\end{cases}.$$
 Then we have
$$s_{ij}=\alpha_iA^2\alpha_j^T=(\alpha_iA)(A\alpha_j^T)=\sum_{z=1}^b c_{iz} r_{zj }n_z =(CNR)_{ij}.$$ 
\end{proof}

From (\ref{1}) it follows that
$W_{ij}=\delta_{ij}(t-\mu)I+\mu J+(\lambda-\mu) A_{ij}$, and therefore
\begin{align*}
	s_{ij}=&\delta_{ij}(t-\mu)n_j+\mu n_i n_j +(\lambda-\mu)c_{ij}n_j\\
	=&\delta_{ij}(t-\mu)n_i+\mu n_i n_j +(\lambda-\mu)r_{ij}n_i.
\end{align*}

Furthermore, from  (\ref{3}) and Proposition \ref{cnr} we have  
	$$s_{ij}=\sum_{z=1}^{b}c_{iz}r_{zj}n_z
	=\sum_{z=1}^{b}c_{iz}c_{zj}n_j
	=\sum_{z=1}^{b}r_{iz}r_{zj}n_i.$$
It follows that the elements of the column orbit matrix $C$ fulfil the  following conditions
$$\delta_{ij}(t-\mu)n_j+\mu n_i n_j +(\lambda-\mu)c_{ij}n_j=\sum_{z=1}^{b}c_{iz}c_{zj}n_j,$$
that is
$$\delta_{ij}(t-\mu)+\mu n_i +(\lambda-\mu)c_{ij}=\sum_{z=1}^{b}c_{iz}c_{zj},$$
while for the elements of the row orbit matrix $R$ applies in a similar way 
$$\delta_{ij}(t-\mu)+\mu n_j +(\lambda-\mu)r_{ij}=\sum_{z=1}^{b}r_{iz}r_{zj}.$$
 
Therefore, we introduce the following definitions of the orbit matrices of DSRGs (see \cite{csz}, Definitions 3.3 and 3.4).

\begin{definition} \label{row}
	A $(b \times b)$-matrix $R = [r_{ij}]$ with entries satisfying conditions:
	\begin{align}
		\sum_{j=1}^b r_{ij}&=\sum_{i=1}^b \frac{n_i}{n_j} r_{ij}=k  \label{s1}\\
		\sum_{z=1}^{b}r_{iz}r_{zj}=&\delta_{ij}(t-\mu)+\mu n_j +(\lambda-\mu)r_{ij} \label{s2}
	\end{align}
	where  $0\leq r_{ij}\leq n_j$, $0\leq r_{ii}\leq n_i-1$ and $\sum_{i=1}^{b}n_i=v$,
	is called a {\bf row orbit matrix} for a directed strongly regular graph 
	with parameters $(v,k,t, \lambda, \mu)$ and the orbit lengths distribution $(n_1, \ldots,n_b)$.
\end{definition}

\begin{definition}\label{col}
	A $(b \times b)$-matrix $C = [c_{ij}]$ with entries satisfying conditions:
	\begin{align}
		\sum_{i=1}^b c_{ij}&=\sum_{j=1}^b \frac{n_j}{n_i} c_{ij}=k  \label{s3}\\
		\sum_{z=1}^{b}c_{iz}c_{zj} =&\delta_{ij}(t-\mu)+\mu n_i +(\lambda-\mu)c_{ij}\label{s4}
	\end{align}
	where  $0\leq c_{ij}\leq n_i$, $0\leq c_{ii}\leq n_i-1$ and $\sum_{i=1}^{b}n_i=v$, is called a {\bf column orbit matrix} 
	for a directed strongly regular graph 
	with parameters $(v,k,t, \lambda, \mu)$ and the orbit lengths distribution $(n_1, \ldots, n_b)$.
\end{definition}
\begin{remark}
When $k=t$ then a DSRG$(v,k,t, \lambda, \mu)$ is a SRG$(v,k, \lambda, \mu)$ and definitions \ref{col} and \ref{row} corresponds to definitions 3.1. and 3.2. from \cite{cm}.
\end{remark}

\section{A method of the construction of DSRGs} \label{con}

The table of feasible parameters of directed strongly regular graphs is available at \cite{bro1}.
In our construction we assume that a group $G$ is an automorphism group of a directed strongly regular graph $\Gamma=(V,E)$ with parameters $(v,k,t, \lambda, \mu)$ and partitions the vertex set $V$ of $\Gamma$ into \( b \) orbits, denoted by \( O_1, O_2, \ldots, O_b \), where \( n_i=|O_i| \) for \( i \in \{1, 2, \ldots, b \}\).

We start the construction by determining all possible orbit lengths distributions $(n_1,n_2,...n_b)$ for the action of a group $G$ on $\Gamma$, keeping in mind that \( n_i \) divides \( |G| \) for $ i \in \{1, 2, \ldots, b\} $ and that the sum of the sizes of all orbits must be equal to the size of the vertex set $V$, that is 
$\sum_{i=1}^b n_i=v.$
The next step is to find prototypes for rows and columns of column orbit matrices of directed strongly regular graphs. The concept of a prototype for a row of a column orbit matrix of a strongly regular graph with a presumed automorphism group of prime order was introduced in \cite{beh} and extended to composite order automorphisms acting on strongly regular graphs in \cite{cm}. Here we further generalize this concept and introduce prototypes for rows and columns of column orbit matrices of DSRGs.

\subsection{Prototypes for rows and columns of a column orbit matrix of a directed strongly regular graph}\label{sec-old}

A prototype for a row (column) of a column orbit matrix $C$ tells us the frequency of each integer that appears as an entry in a particular row (column) of $C$, i.e. it tells us how often each entry occurs without taking into account the order of the entries. Columns or rows of $C$ that correspond to orbits of size $1$ are said to be {\it fixed}, while all others are said to be {\it nonfixed}.

With $l_i$ , $i = 1, . . . , \rho$, we denote different elements of the multiset $\{n_1,n_2,...n_b\}$ in ascending order $l_1<l_2<\ldots<l_{\rho}$. Let $d_{l_i}$ be the number of orbits of length $l_i$, $i=1,\dots,\rho,$ for an action of $G$ on $\Gamma$.

\subsubsection{Prototypes for fixed rows and fixed columns}
Let $r$ be a fixed row (fixed column) corresponding to the orbit $O_r$ of size $n_r=1$. Then $l_1=1$.  

With $y_e^{(l_i)}$ we denote the number of occurrences of an element  $e\in \{0,m\}$ 
at the positions of row $r$ (column $r$) that correspond to the orbits of length $l_i,i=1,\dots,\rho$, where
$$m=\begin{cases}
		1, \quad r \textnormal{ is a fixed row}\\
		l_i, \quad r \textnormal{ is a fixed column}
	\end{cases}.$$ Note that from $c_{rr}=0$, we have $y_0^{(1)}\geq 1.$ Moreover,
\begin{equation}
	\label{dva+}
y_0^{(l_i)}+y_m^{(l_i)}=d_{l_i}\end{equation}
 \begin{equation}\label{tri+}\sum_{i=1}^{\rho} l_i y_m^{(l_i)}=k,\end{equation}
 where $i=1,\dots,\rho$. 
 The vector $$(y_0^{(1)},y_m^{(1)};y_0^{(l_2)},y_m^{(l_2)};\dots;y_0^{(l_{\rho})},y_m^{(l_{\rho})})$$ whose components are nonnegative integer solutions of the equations (\ref{dva+}) and (\ref{tri+}) 
is called a {\it prototype for a fixed row} ({\it prototype for a fixed column}).

\subsubsection{Prototypes for nonfixed rows and  nonfixed columns}

Let  $r$ be a nonfixed row (nonfixed column) corresponding  to the orbit $O_r$  of size $n_r\neq 1$.

With $y_e^{(l_i)}$ we denote the number of occurrences of an element  $e\in \{0,1,\dots,m\}$
at the positions of row $r$ (column $r$) corresponding to the orbits of length $l_i,i=1,\dots,\rho$, where
$$m=\begin{cases}
		n_r, \quad r \textnormal{ is a nonfixed row}\\
		l_i, \quad r \textnormal{ is a nonfixed column}
	\end{cases}.$$ 
For $e\frac{n_i}{n_r}\notin \{0,1,\dots, n_i \}$ ($e\frac{n_r}{n_i}\notin \{0,1,\dots, n_r \}$), we have $y_{e}^{(l_i)}=0.$ Note that if $l_1=1$, then  $y_e^{(1)}=0$ for $e\notin \{0,m\}$.

The following applies:
	\begin{equation}		\label{11a}
	\sum_{e=0}^{m}y_e^{(l_i)}=d_{l_i},  i=1,\dots,\rho.
\end{equation}
Since the sum of each row and column of the adjacency matrix is $k$, we also have
\begin{equation}
	\label{12a}
	\sum_{i=1}^{\rho}\sum_{h=1}^{m}y_h^{(l_i)}h\frac{{l_i}}{m}=k.\end{equation}
The vector $$(y_0^{(l_1)},\dots, y_{m}^{(l_1)};y_0^{(l_2)},\dots,y_{m}^{(l_2)};\dots;y_0^{(l_{\rho})},\dots,y_{m}^{(l_{\rho})}),$$ 
whose components are nonnegative integer solutions of the equations (\ref{11a}) and (\ref{12a}) 
is called a {\it prototype for a row corresponding to the orbit of length $n_r$} ({\it prototype for a column corresponding to the orbit of length $n_r$}). 

If $l_1=1,$ then  $y_e^{(1)}=0$ for $e\notin \{0,m\}$, and in this case we will be looking for the vector
 $$(y_{0}^{(1)},y_{m}^{(1)};y_0^{(l_2)},\dots,y_{m}^{(l_2)};\dots;y_0^{(l_{\rho})},\dots,y_{m}^{(l_{\rho})}).$$ 

\subsection{Construction of DSRGs}

Using the prototypes introduced in the previous section, we construct column orbit matrices for a group $G$ which is an automorphism group of a tentative directed strongly regular graph $\Gamma=(V,E)$ with parameters $(v,k,t, \lambda, \mu)$ that partitions
 the vertex set $V$ of $\Gamma$ into \( b \) orbits, denoted by \( O_1, O_2, \ldots, O_b \), where \( n_i=|O_i| \) for \( i \in \{1, 2, \ldots, b \}\). The construction procedure is similar to the construction of other incidence structures. However, when constructing other incidence structures, usually only prototypes of rows of column orbit matrices are used (see, for example, \cite{beh,cm,c-r-metrika}), and here we need to check the validity of the equation (\ref{s4}) at each construction step. Therefore, we alternately construct rows and columns of a column orbit matrix and check at each step whether the equation (\ref{s4}), that is
\begin{equation*}
\sum_{z=1}^{b}c_{iz}c_{zj}=
\delta_{ij}(t-\mu)+\mu n_i +(\lambda-\mu)c_{ij},\end{equation*}
is satisfied. We use the row and column prototypes introduced in the previous section to construct the rows and columns of a column orbit matrix. During construction, the number of candidate matrices for the column orbit matrix can be reduced by considering the orbits of equal length that have conjugate stabilizers to avoid the construction of isomorphic structures (see  \cite{cm,c-r-metrika}).

Several situations can occur in the construction: 1) there are no orbit matrices, 2) column orbit matrices are constructed, and 3) the situation is inconclusive because the construction could not be performed due to the large number of possibilities. In the last case, one can try to work with a larger group of composite order and perform an additional refinement step before extending the orbit matrices to adjacency matrices of DSRGs (see, for example, \cite{cm,c-r-metrika}).

The final step in the construction of DSRGs is a process called indexing, in which the obtained orbit matrices are extended to adjacency matrices of DSRGs. This refinement of the orbit matrices follows an approach that has already been successfully used in the construction of other incidence structures (see for example, \cite{cm,c-r-metrika}). When working with DSRGs, we again switch between rows and columns and check all necessary conditions for the adjacency matrix of a DSRG in each step. To reduce the computational burden, we sometimes approach to indexing partially constructed orbit matrices, which can be particularly useful in cases where the corresponding DSRG does not exist and trying to construct an orbit matrix is a tedious and not necessarily successful process. We will demonstrate this approach in the next section (see Subsection \ref{sub_aut7}). Note that not every orbit matrix necessarily leads to a directed strongly regular graph. Moreover, a given orbit matrix can correspond to several nonisomorphic directed strongly regular graphs.

\begin{remark}
If a DSRG is a strongly regular graph, that is if $k=t$, then due to $c_{zj}{n_j}=c_{jz}n_z$, the condition (\ref{s4}) becomes $\displaystyle \sum_{z=1}^{b}\frac{n_z}{n_j}c_{iz}c_{jz} =\delta_{ij}(k-\mu)+\mu n_i +(\lambda-\mu)c_{ij}$ and it is sufficient to work with row prototypes. In this case, our construction goes back to the construction of strongly regular graphs described in \cite{beh,cm}.
\end{remark}

\section{DSRGs$(22,9,6,3,4)$ with nontrivial automorphism}\label{22}

The smallest parameter set for which the existence of a directed strongly regular graph is unknown is $(22,9,6,3,4)$ (see \cite{bro1}). Here we use the method for constructing DSRGs described in the previous section to study DSRGs$(22,9,6,3,4)$ that admit a nontrivial automorphism. To this end, we consider the automorphisms of prime order that can act on such a graph. As a result, we prove the existence of a directed strongly regular graph with parameters $(22,9,6,3,4)$ and classify all such graphs with an automorphism of order three. Furthermore, we construct all DSRGs$(22,9,6,3,4)$ with an automorphism of order two that fixes six or eight vertices.

Let $\rho$ be an automorphism of prime order $p$ acting on a DSRG with parameters $(22,9,6,3,4)$ and denote by $f_p$ the number of fixed vertices for such an action. Since $\rho$ is of prime order, the possible orbit lengths are $l_1=1$ and $l_2=p$, and the orbit lengths distributions are determined by the number of orbits of length one $d_1=f_p$. In this section we will prove the following theorem.

\begin{theorem} \label{fp}
Let $\rho$ be an automorphism of prime order $p$ acting on a DSRG$(22,9,6,3,4)$ and denote by $f_p$ the number of fixed vertices for such an action.  Then $p \in \{2,3\}$ and the following holds:
\begin{enumerate}
 \item $f_2 \in \{0,6,8,10,12,14,16 \}$,
 \item $f_3=1$.
\end{enumerate}

\end{theorem}

To prove Theorem \ref{fp}, we will consider all possible orbit lengths distributions for an automorphism of prime order $p$ acting on a DSRG$(22,9,6,3,4)$. We have to consider the cases that satisfy $2\leq p\leq 19$, $d_1\leq 21$ and $f_p+d_p\cdot p =22$.

 \subsection{Automorphisms of order 13, 17 and 19}
 
Let $\rho$ be an automorphism of prime order $p \in \{13,17,19 \}$ acting on a DSRG$(22,9,6,3,4)$ and denote by $f_p$ the number of fixed vertices for such an action. Then $f_{13}=9$, $f_{17}=5$ and $f_{19}=3$. In all these cases $d_p=1$. 
 
\begin{lemma} \label{aut13}
Let $\rho$ be an automorphism of order $p$ acting on a DSRG$(22,9,6,3,4)$. Then $p \notin \{13,17,19 \}$. 
\end{lemma} 
 
\begin{proof}
We are looking for the prototypes for fixed rows of the column orbit matrices in the cases $p \in \{13,17,19 \}$. To do this, we need to find the nonnegative integer solutions of the equation (\ref{tri+}).\\
For $p=13$, the only nonnegative integer solution of $y_1^{(1)}+13y_1^{(13)}=9$ is given by $y_1^{(1)}$=9, $y_1^{(13)}$=0. Since $y_0^{(1)}\geq 1$ and $y_0^{(1)}+y_1^{(1)}=9$, we also have $y_1^{(1)}<9$. In this case, there is no prototype for a fixed row.\\
For $p=17$ and $p=19$ there are no nonnegative integer solutions for $y_1^{(1)}+17y_1^{(17)}=9$, $y_0^{(1)}+y_1^{(1)}=5$
and $y_1^{(1)}+19y_1^{(19)}=9$, $y_0^{(1)}+y_1^{(1)}=3$, respectively.
\end{proof}

 \subsection{An action of an automorphism of order 11}

 \begin{lemma} \label{aut11}
A DSRG$(22,9,6,3,4)$ with an automorphism of order $11$ does not exist. 
\end{lemma} 

\begin{proof}
We have to consider two cases: $f_{11}=0$ and $f_{11}=11$.

If $f_{11}=0$, we have only prototypes for the rows and columns corresponding to the orbits of length $11$. Those are the solutions of 
 \begin{gather*}
 \sum_{h=0}^{11}y_h^{(11)}=2\\
\sum_{h=1}^{11}y_h^{(11)}h=9\\
 \end{gather*}
We get that a prototype $(y_0^{(11)},y_1^{(11)},y_2^{(11)},y_3^{(11)},y_4^{(11)},y_5^{(11)},y_6^{(11)},y_7^{(11)},y_8^{(11)},y_9^{(11)},y_{10}^{(11)},y_{11}^{(11)})$ must be from the set  $\{(1,0,0,0,0,0,0,0,0,1,0,0), (0,1,0,0,0,0,0,0,1,0,0,0),  (0,0,1,0,0, 0,0,1,0,0,\\0,0),(0,0,0,1,0,0,1,0,0,0,0,0), (0,0,0,0,1,1,0,0,0,0,0,0)\}$. 
These possibilities for the first row of a column orbit matrix determine the rest of the matrix and we get the following cases (in the first row and the first column we give the length of the corresponding orbit).
  \begin{center} \footnotesize
 	\begin{table}[htpb!] 
\begin{minipage}{.195\textwidth}
 		\[ { \begin{tabular}{c||cc}
 				&11&11\\
 				\hline
 				\hline
 				11&0&9\\
				11&9&0
 			
 		\end{tabular}}\] 
\end{minipage}
\begin{minipage}{.195\textwidth}
 		\[ { \begin{tabular}{c||cc}
 				&11&11\\
 				\hline
 				\hline
 				11&1&8\\
 				11&8&1\\
 		\end{tabular}}\] 
\end{minipage}
\begin{minipage}{.195\textwidth}
		\[ { \begin{tabular}{c||cc}
 				&11&11\\
 				\hline
 				\hline
 				11&2&7\\
 				11&7&2\\
 		\end{tabular}}\] 
\end{minipage}
\begin{minipage}{.195\textwidth}
		\[ { \begin{tabular}{c||cc}
 				&11&11\\
 				\hline
 				\hline
 				11&3&6\\
 				11&6&3\\
 		\end{tabular}}\] 
\end{minipage}
\begin{minipage}{.195\textwidth}
		\[ { \begin{tabular}{c||cc} 
 				&11&11\\
 				\hline
 				\hline
 				11&4&5\\
 				11&5&4\\
 		\end{tabular}}\] 
\end{minipage}
 	\end{table}
 \end{center} 
None of these matrices satisfies the condition (\ref{s4}) and is not the column orbit matrix of a DSRG$(22,9,6,3,4)$.

Let us now consider the case $f_{11}=11$. The only solution of the equations
 \begin{gather*}
 y_0^{(1)}+y_1^{(1)}=11\\
 y_0^{(11)}+y_1^{(11)}=1\\
 y_1^{(1)}+11y_1^{(11)}=9
 \end{gather*}
 is $(y_0^{(1)},y_1^{(1)};y_0^{(11)},y_1^{(11)})=(2,9,1,0).$
Without loss of generality, we obtain that the first fixed row of a column orbit matrix is given by
  \begin{center} \footnotesize
 	\begin{table}[htpb!] 
 		\[ { \begin{tabular}{c||ccccc ccccc cc}
 				$OM$&1&1&1&1&1&1&1&1&1&1&1&11\\
 				\hline
 				\hline
 				1&0&0&1&1&1&1&1&1&1&1&1&0
 		\end{tabular}}\] 
 	\end{table}
 \end{center} 
For the prototype for a fixed column is also
 $(y_0^{(1)},y_1^{(1)};y_0^{(11)},y_{11}^{(11)})=(2,9,1,0)$, it is not possible to construct the first column such that the product of the first row and the first column is equal to $t=6$. Therefore, (\ref{s4}) is not fulfilled and the statement of the lemma applies.
 \end{proof}

 \subsection{An action of an automorphism of order 7} \label{sub_aut7}
   
 \begin{lemma} \label{aut7}
A DSRG$(22,9,6,3,4)$ with an automorphism of order $7$ does not exist. 
\end{lemma} 

\begin{proof}
We have three cases to consider: $f_7=1, f_7=8$ and $f_7=15$.

For $f_7=1$ there are no nonnegative integer solutions satisfying $y_1^{(1)}+7y_1^{(7)}=9$ and $y_1^{(1)}=0$. Therefore, the prototypes for the fixed row and the fixed column do not exist.

For $f_7=8$, the only prototype for a fixed row is $(y_0^{(1)},y_1^{(1)};y_0^{(7)},y_1^{(7)})=(6,2,1,1)$ and we can assume that the first row of an orbit matrix is given as follows.
 
  \begin{center} \footnotesize
 	\begin{table}[htpb!] 
 		\[ { \begin{tabular}{c||ccccc ccccc }
 				$OM$&1&1&1&1&1&1&1&1&7&7\\
 				\hline
 				\hline
 				1&0&0&0&0&0&0&1&1&0&1
 		\end{tabular}}\] 
 	\end{table}
 \end{center} 
Since the only prototype for the fixed column is also
 $(y_0^{(1)},y_1^{(1)};y_0^{(7)},y_7^{(7)})=(6,2,1,1)$, it is not possible to construct the first column of a column orbit matrix such that the product of the first row and the first column equals $t=6.$

Finally, let us consider the case $f_7=15$. There is only one orbit of length $7$ and we start our construction of a column orbit matrix with the nonfixed row and the nonfixed column. The prototypes for the nonfixed row and the nonfixed column belong to the set $\{(13, 2, 0, 0, 0, 0, 0, 0, 0, 1), (12, 3, 0, 0, 0, 0, 0, 0, 1, 0 ), (11, 4, 0, 0, 0, 0, 0, 1, 0, 0 ), ( 10, 5,0,0, 0, 0, 1, 0,\\ 0, 0), ( 9, 6, 0, 0, 0, 1, 0, 0, 0,
0 ), ( 8, 7, 0, 0, 1, 0, 0, 0, 0, 0), ( 7, 8, 0, 1, 0, 0, 0, 0, 0, 0), ( 6, 9, 1, 0, 0, 0, 0, 0,\\ 0, 0 )\}$. However, the only pair of prototypes for the nonfixed row and the nonfixed column that satisfies the equation (\ref{s4})  
\begin{equation*}
\sum_{z=1}^{b}c_{iz}c_{zj}=
\delta_{ij}(t-\mu)+\mu n_i +(\lambda-\mu)c_{ij},\end{equation*}
is the pair consisting of 
 \begin{center} \footnotesize
 	\begin{table}[htpb!] 
 		\[ { \begin{tabular}{c||cccccccccccccccc }
		$OM$&1&1&1&1&1&1&1&1&1&1&1&1&1&1&1&7\\
\hline	
\hline
7&0&0&0&0&0&0&0&7&7&7&7&7&7&7&7&1
 		\end{tabular}}\] 
 	\end{table}
 \end{center} 
 as the nonfixed row, and  
  \begin{center} \footnotesize
 	\begin{table}[htpb!] 
 		\[ { \begin{tabular}{c||cccccccccccccccc }
		$OM^T$&1&1&1&1&1&1&1&1&1&1&1&1&1&1&1&7\\
\hline	
\hline
7&0&0&0&1&1&1&1&0&0&0&0&1&1&1&1&1\\
 		\end{tabular}}\] 
 	\end{table}
 \end{center} 
as the nonfixed column of a column orbit matrix.
 
Now we apply indexing to this partial orbit matrix, which consists of a nonfixed row and a nonfixed column. The element $c_{ij}={\bm a}$ of the orbit matrix is transformed in the following way, depending on whether it is in a fixed or nonfixed row and a fixed or nonfixed column:
 \begin{center}
 \begin{tabular}{|c|c|}\cline{2-2}
\multicolumn{1}{c|}{} &$n_j=1$\\ \hline
\multirow{1}{*}{$n_i=1$} &${\bm a}$ \\ \hline
\end{tabular} \hspace{2 cm}
 \begin{tabular}{|c|c|c|c|c|c|c|c|}\cline{2-8}
\multicolumn{1}{c|}{} &\multicolumn{7}{c|}{$n_j=7$}\\ \hline
\multirow{1}{*}{$n_i=1$} &${\bm a}$&${\bm a}$&${\bm a}$&${\bm a}$&${\bm a }$&${\bm a}$&${\bm a}$ \\ \hline
\end{tabular}
\vspace{1cm}

 \begin{tabular}{|c|c|}\cline{2-2}
\multicolumn{1}{c|}{} &$n_j=1$\\ \hline
\multirow{7}{*}{$n_i=7$} &${\bm a/n_i}$ \\ \cline{2-2}
 &${\bm a/n_i}$ \\ \cline{2-2}
 &${\bm a/n_i}$ \\ \cline{2-2}
 &${\bm a/n_i}$ \\ \cline{2-2}
 &${\bm a/n_i}$ \\ \cline{2-2}
 &${\bm a/n_i}$ \\ \cline{2-2}
 &${\bm a/n_i}$ \\ \hline
\end{tabular} \hspace{2 cm}
\begin{tabular}{|c|c|c|c|c|c|c|c|}\cline{2-8}
\multicolumn{1}{c|}{}&\multicolumn{7}{c|}{$n_j=7$}\\ \hline
\multirow{7}{*}{$n_i=7$} &${\bm a_1}$&${\bm a_2}$&${\bm a_3}$&${\bm a_4}$&${\bm a_5}$&${\bm a_6}$&${\bm a_7}$ \\ \cline{2-8}
&${\bm a_7}$&${\bm a_1}$&${\bm a_2}$&${\bm a_3}$&${\bm a_4}$&${\bm a_5}$&${\bm a_6}$ \\ \cline{2-8}
&${\bm a_6}$&${\bm a_7}$&${\bm a_1}$&${\bm a_2}$&${\bm a_3}$&${\bm a_4}$&${\bm a_5}$ \\ \cline{2-8}
&${\bm a_5}$&${\bm a_6}$&${\bm a_7}$&${\bm a_1}$&${\bm a_2}$&${\bm a_3}$&${\bm a_4}$ \\ \cline{2-8}
&${\bm a_4}$&${\bm a_5}$&${\bm a_6}$&${\bm a_7}$&${\bm a_1}$&${\bm a_2}$&${\bm a_3}$ \\ \cline{2-8}
&${\bm a_3}$&${\bm a_4}$&${\bm a_5}$&${\bm a_6}$&${\bm a_7}$&${\bm a_1}$&${\bm a_2}$ \\ \cline{2-8}
&${\bm a_2}$&${\bm a_3}$&${\bm a_4}$&${\bm a_5}$&${\bm a_6}$&${\bm a_7}$&${\bm a_1}$ \\ \hline
\multicolumn{1}{c}{}&\multicolumn{7}{c}{$\sum_{i=1}^{7} a_i=a$, $a_i\in\{0,1\}$} \\
\end{tabular}

\end{center}

In this way, a nonfixed row and a nonfixed column of an orbit matrix expand to seven rows and seven columns of the adjacency matrix of a DSRG$(22,9,6,3,4)$. Since an adjacency matrix of a DSRG can be considered as an orbit matrix for an action of the trivial automorphism group, these rows and columns should satisfy (\ref{s4}). Since this is not the case, we can conclude that there is no DSRG$(22,9,6,3,4)$ with an automorphism of order $7$.
\end{proof}

\subsection{An action of an automorphism of order 5} 
For an automorphism of order $5$ acting on a DSRG$(22,9,6,3,4)$, we have $f_5 \in \{2,7,12,17 \}$.

 \begin{lemma} \label{aut5}
A DSRG$(22,9,6,3,4)$ with an automorphism of order $5$ does not exist. 
\end{lemma} 

\begin{proof}
We eliminate the case $f_5=2$, since there are no nonnegative integer solutions that satisfy $y_1^{(1)}+5y_1^{(5)}=9$ and $y_1^{(1)}\leq 1$.

 For the case $f_5=7$, we obtain that the only prototype for a fixed row is $(y_0^{(1)},y_1^{(1)};y_0^{(5)},y_1^{(5)})=(3,4,2,1)$, and we can therefore assume that the first row of a column orbit matrix is given by
 
  \begin{center} \footnotesize
 	\begin{table}[htpb!] 
 		\[ { \begin{tabular}{c||ccccc ccccc }
 				$OM$&1&1&1&1&1&1&1&5&5&5\\
 				\hline   \hline
 				1&0&0&0&1&1&1&1&0&0&1
 		\end{tabular}}\] 
 	\end{table}
 \end{center} 
The only prototype for a fixed column is also
 $(y_0^{(1)},y_1^{(1)};y_0^{(5)},y_5^{(5)})=(3,4,2,1)$ and it is easy to verify that it is not possible to construct a fixed column such that (\ref{s4}) is satisfied. 
 
If $f_5=12$, then prototypes for fixed rows and prototypes for fixed columns must be from the set $\{(3,9,2,0),(8,4,1,1)\}$. The only pair that satisfies (\ref{s4}) is the one where both prototypes are equal to $(8,4,1,1)$. We must therefore construct $12$ rows and $12$ columns of a column-orbit matrix using only the prototype $(8,4,1,1)$. However, it is possible to construct the first three rows and the first two columns of an orbit matrix, and then we cannot construct the third column such that (\ref{s4}) is satisfied. Therefore, there is no DSRG(22,9,6,3,4) with an automorphism of order $5$ that fixes $12$ vertices.

For the case $f_5=17$ we have constructed a partial orbit matrix consisting of the nonfixed row and the nonfixed column and tried to index it. There is only one such partial orbit matrix up to isomorphism, with the nonfixed row and the nonfixed column as follows. 
 \begin{center} \footnotesize
 	\begin{table}[htpb!] 
	\[ { \begin{tabular}{c||cccccccccccccccccc}%
$OM$&1&1&1&1&1&1&1&1&1&1&1&1&1&1&1&1&1&5\\ \hline%
5&0&0&0&0&0&0&0&0&0&5&5&5&5&5&5&5&5&1
 		\end{tabular}}\] 
		
\[ { \begin{tabular}{c||cccccccccccccccccc}%
$OM^T$&1&1&1&1&1&1&1&1&1&1&1&1&1&1&1&1&1&5\\ \hline%
5&0&0&0&0&0&1&1&1&1&0&0&0&0&1&1&1&1&1
 		\end{tabular}}\] 
 			 	\end{table}
 \end{center} 
Similarly as described in the case with an automorphism of order $7$ with $f_7=15$, we come to the conclusion that there is no DSRG$(22,9,6,3,4)$ with automorphism of order $5$ with $f_5=17$. This concludes the proof of the lemma.
\end{proof}

\subsection{Classification of DSRGs$(22,9,6,3,4)$ with an automorphism of order 3} 

In the previous sections we have shown that a DSRG$(22,9,6,3,4)$ with an automorphism of prime order $p$ where $p>3$ does not exist. In this section we will prove the existence of a directed strongly regular graph with parameters $(22,9,6,3,4)$ with an automorphism of order $3$ and classify all such graphs up to isomorphism. For the construction of orbit matrices and their indexing we use our programs written for GAP \cite{GAP}. GAP has also been used for isomorphism tests and when working with automorphism groups.

\begin{lemma} \label{fp_3}
If an automorphism of order three acts on a DSRG$(22,9,6,3,4)$, then it fixes exactly one vertex.
\end{lemma}

\begin{proof}
The possible cases are given by $f_3\in\{1,4,7,10,13,16,19\}$.
 In all these cases, it is possible to construct prototypes for fixed and nonfixed rows and columns. The number of prototypes for each case is specified in Table \ref{prototypes}.
\begin{center} \footnotesize
	\begin{table}[htpb!] 
		\caption{ \label{prototypes} The number of prototypes for fixed and nonfixed rows and columns for an automorphism of order $3$ acting on a DSRG$(22,9,6,3,4)$}
\[ { \begin{tabular}{c||c|c|c|c|c|c|c|c|}
\hline 
$f_p$&1&4&7&10&13&16&19\\
\hline
$\#$prototypes for fixed rows (columns)&1&2&3&4&4&3&2\\
\hline
$\#$prototypes for nonfixed rows (columns)& 19 & 35&38&31  &20&10&4
\cr \hline 
		\end{tabular}}\] 
	\end{table}
\end{center} 

To prove the statement of the lemma, we must prove $f_3\notin\{4,7,10,13,16,19\}$.

If $f_3=4$, we have to construct $4$ rows and $4$ columns of a column orbit matrix using prototypes for fixed rows and fixed columns. However, once a fixed row and a fixed column of a column orbit matrix have been constructed, it is not possible to construct the second fixed row in such a way that (\ref{s4}) is satisfied. Similarly, if $f_3=7$, then it is only possible to construct two fixed rows and two fixed columns of a column orbit matrix, while it is not possible to construct a third fixed row. These two cases are therefore not possible.

For $f_3=10$ we have constructed $12$ column orbit matrices. When indexed, they all did not result in DSRG$(22,9,6,3,4)$. So there is no DSRG$(22,9,6,3,4)$ that admits an automorphism of order three with $f_3=10$.

As the number of fixed vertices increases, the construction becomes more computationally demanding. Therefore, for the cases $f_3 \in \{13,16,19 \}$ we start with the construction of a column orbit matrix with nonfixed rows and nonfixed columns and apply the indexing to partial orbit matrices as described in Section \ref{sub_aut7} (see the case with an automorphism of order $7$ with $f_7=15$). With this approach, we were able to eliminate some cases without trying to construct complete column orbit matrices and extend them to the adjacency matrix of a DSRG$(22,9,6,3,4)$ using computers, and reduce the computations in the remaining cases. Here we give more details on each of the cases.\\ 
Considering all pairs of prototypes for nonfixed rows and nonfixed columns for $f_3=13$, as described in Section \ref{sub_aut7}, reduced the number of prototypes to be used in a construction from $20$ to $6$. Our computations have shown that there are no column orbit matrices that can be formed from these prototypes and $4$ prototypes for fixed rows and fixed columns.\\
Our analysis has shown that for $f_3=16$ and $f_3=19$ there are no pairs consisting of a nonfixed row and a nonfixed column whose indexing would lead to rows and columns of the adjacency matrix that satisfy (\ref{s4}). Therefore, we can exclude these cases. This concludes the proof of the lemma.
\end{proof}

In the next theorem, we prove the existence of a DSRG$(22,9,6,3,4)$ by considering the remaining case $f_3=1$.
 
\begin{theorem} \label{th_aut3_1}
There are up to isomorphism exactly $134$ directed strongly regular graphs with parameters $(22,9,6,3,4)$ with an automorphism of order three, of which $14$ have the full automorphism group of order $6$ isomorphic to the symmetric group $S_3$, while for the others the full automorphism group is of order three.
\end{theorem}

\begin{proof}

Due to Lemma \ref{fp_3} we only have to consider the case $f_3=1$. As indicated in Table \ref{prototypes}, there is only one prototype for the fixed row and the fixed column, and we can assume that the first row and the first column of a column orbit matrix are given as follows.
 \begin{center} \footnotesize
\begin{table} [htpb!]
\[ { \begin{tabular}{c||cccccccc}
$OM$&1&3&3&3&3&3&3&3\\ \hline
1&0&0&0&0&0&1&1&1\\
\end{tabular}}\] 

\[ { \begin{tabular}{c||cccccccc}
$OM^T$&1&3&3&3&3&3&3&3\\ \hline
1&0&0&0&0&3&0&3&3\\
\end{tabular}}\] 
\end{table}
\end{center}
Furthermore, there are $19$ prototypes for nonfixed rows and nonfixed columns. Using these prototypes, we constructed $883$ column orbit matrices, of which $46$ lead to DSRGs$(22,9,6,3,4)$ through the process of indexing. We constructed $134$ nonisomorphic DSRGs$(22,9,6,3,4)$. Further analysis showed that $14$ of them have the full automorphism group of order $6$ isomorphic to the symmetric group $S_3$, while for the others the full automorphism group is of order three. The adjacency matrices of $134$ constructed DSRGs$(22,9,6,3,4)$ are available online at
 \begin{verbatim}
http://www.math.uniri.hr/~mmaksimovic/DSRG(22,9,6,3,4)_z3.pdf
\end{verbatim}     
\end{proof}

\subsection{DSRGs$(22,9,6,3,4)$ with an automorphism of order 2} 

In the previous sections we proved the existence of a DSRG$(22,9,6,3,4)$ and completed a classification of DSRGs$(22,9,6,3,4)$ with automorphisms of odd prime order. Since $14$ of the constructed DSRGs$(22,9,6,3,4)$ have the symmetric group $S_3$ of order $6$ as the full automorphism group, we know that there exist DSRGs$(22,9,6,3,4)$ with an automorphism of order $2$. Here we examine DSRGs$(22,9,6,3,4)$ with an automorphism of order $2$.

\begin{lemma} \label{fp_2}
Let $f_2$ be the number of fixed points of an automorphism of order two acting on a DSRG$(22,9,6,3,4)$. Then it holds that $f_2 \in \{0,6,8,10,12,14,16 \}$.
\end{lemma}

\begin{proof}
The possible cases are  $f_2\in\{0,2,4,6,8,10,12,14,16,18,20\}$ and the number of prototypes for fixed and nonfixed rows and columns of column orbit matrices in each case is given in Table \ref{prototypes2}.
\begin{center} \footnotesize
	\begin{table}[htpb!] 
		\caption{ \label{prototypes2} The number of prototypes for fixed an nonfixed rows and columns for an automorphism of order $2$ acting on a DSRG$(22,9,6,3,4)$}
\[ { \begin{tabular}{c||c|c|c|c|c|c|c|c|c|c|c|c}
\hline 
$f_2$&0&2&4&6&8&10&12&14&16&18&20\\
\hline
$\#$prototypes for fixed rows (columns)&0&1&2&3&4&5&5&5&4&3&2\\
\hline
$\#$prototypes for nonfixed rows (columns)& 5 & 14&21&25  &26&24&20&15&10&6&3
\cr \hline 
		\end{tabular}}\] 
	\end{table}
\end{center} 
For the case $f_2=2$, there is only one prototype for fixed rows and fixed columns, and we can assume that the first row of a column orbit matrix is given as follows.
   \begin{center} \footnotesize
 	\begin{table}[htpb!] 
 		\[ { \begin{tabular}{c||ccccc ccccc cc }
 				 
 				$OM$&1&1&2&2&2&2&2&2&2&2&2&2\\
 				\hline
 				1&0&1&0&0&0&0&0&0&0&1&1&1\\
 			
 		\end{tabular}}\] 
 	\end{table}
 \end{center} 
However, it is not possible to construct a fixed column in such a way that the product of this column and the first row equals $6$. The condition (\ref{s4}) is therefore not fulfilled and we can exclude the case $f_2=2$.\\
If $f_2=4$, we have to construct $4$ rows and $4$ columns of a column orbit matrix using prototypes for fixed rows (columns). After constructing two fixed rows and two fixed columns of a column orbit matrix, it is not possible to construct the third fixed row such that (\ref{s4}) is satisfied.

For the cases $f_2 \in \{18,20 \}$ we start with the construction of a column orbit matrix with nonfixed rows and nonfixed columns and apply indexing to partial orbit matrices as described in Section \ref{sub_aut7} (see the case with an automorphism of order $7$ with $f_7=15$). Our analysis has shown that we can exclude both cases, since there are no pairs consisting of a nonfixed row and a nonfixed column whose indexing would yield rows and columns of an adjacency matrix satisfying (\ref{s4}).
\end{proof}

According to Lemma \ref{fp_2}, we have $f_2 \in \{0,6,8,10,12,14,16 \}$. Together with the lemmas \ref{aut13}, \ref{aut11}, \ref{aut7}, \ref{aut5} and \ref{fp_3}, this completes the proof of Theorem \ref{fp}.

It was computationally too demanding for us to solve all remaining cases for an automorphism of order two acting on a DSRG$(22,9,6,3,4)$. However, our computations lead to a classification of DSRGs$(22,9,6,3,4)$ with an automorphism of order two where $f_2 \in \{6,8 \}$. These results are given in the following theorems.

\begin{theorem} \label{th_aut2_6}
Up to isomorphism, there are exactly $82$ directed strongly regular graphs with parameters $(22,9,6,3,4)$ with an automorphism of order two fixing $6$ vertices.
The full automorphism groups of these DSRGs are given in Table \ref{dsrg_2_6}.
 \begin{center} \footnotesize
	\begin{table}[htpb!] 
		\caption{\label{dsrg_2_6} DSRGs(22,9,6,3,4) with  automorphism of order two with $f_2=6$}
		\[ { \begin{tabular}{c||c|c|c|c}
				the full automorphism group &$Z_{2}$&$Z_4$&$Z_2\times Z_2$&$Z_4\times Z_2$\\
				\hline
				\hline
				$  \#DSRGs $&64 & 4&10&4\\
		\end{tabular}}\] 
	\end{table}
\end{center} 
\end{theorem}
\begin{proof}
After constructing $28$ prototypes for the case $f_2=6$ (see Table \ref{prototypes2}), we constructed $264$ column orbit matrices. Our computations showed that only 50 orbit matrices yield to DSRGs$(22,9,6,3,4)$. In total, we obtained $82$ nonisomorphic DSRGs$(22,9,6,3,4)$. Further analysis with GAP led to the results shown in Table \ref{dsrg_2_6}. 
\end{proof}

\begin{theorem} \label{th_aut2_8}
Up to isomorphism, there are exactly $270$ directed strongly regular graphs with parameters $(22,9,6,3,4)$ with an automorphism of order two that fixes eight vertices. Among them, there are $14$ designs with the full automorphism group of order six isomorphic to the symmetric group $S_3$, while the others have the full automorphism group of order two.
\end{theorem}

\begin{proof}
For the case $f_2=8$ we have found prototypes, 26 for nonfixed rows and nonfixed columns, and 4 for fixed rows and fixed columns. These prototypes yield $224$ column orbit matrices, 79 of which can be indexed to obtain directed strongly regular graphs with parameters $(22,9,6,3,4)$. After eliminating isomorphic copies, there are $270$ nonisomorphic DSRGs$(22,9,6,3,4)$ with an automorphism of order two fixing eight vertices. Further analysis showed that $14$ of them have $S_3$ as the full automorphism group, while the others have the full automorphism group of order two. 
 \end{proof}
  
 \section{Conclusion} \label{concl}
 
In this paper, we adapt the well-known method for constructing incidence structures with presumed automorphism group to the construction of directed strongly regular graphs. We applied the developed method to the parameters $(22,9,6,3,4)$, which was the smallest case for which the existence of a directed strongly regular graph was not established. Our results are shown in Table \ref{summ}.

 \begin{center} \footnotesize
\begin{table}[htpb!] 
\caption{\label{summ} Number of DSRGs$(22,9,6,3,4)$ with an automorphism of prime order $p$ and $f_p$ fixed vertices}
\[ { \begin{tabular}{cccc}
$p$ & $f_p$ & $\#$ DSRGs$(22,9,6,3,4)$ & Non-existence proof/Construction\\
\hline 
\hline
2 & 0 & ? & \\
 & 2& 0 & Lemma \ref{fp_2}, no fixed part of a column orbit matrix\\
& 4 & 0& Lemma \ref{fp_2}, no fixed part of a column orbit matrix\\
 & 6 & 82 & Theorem \ref{th_aut2_6}, by exhaustive generation\\
 & 8 & 270 & Theorem \ref{th_aut2_8}, by exhaustive generation\\
 &10 & ? & \\
 &12 & ? & \\
 &14 & ? & \\
 &16 & ? & \\
 &18 & 0 & Lemma \ref{fp_2}, indexing of the nonfixed part not possible\\
 &20 & 0 & Lemma \ref{fp_2}, indexing of the nonfixed part not possible\\
3 & 1 & 134 & Theorem \ref{th_aut3_1}, by exhaustive generation\\
 & 4 & 0 &  Lemma \ref{fp_3}, no fixed part of a column orbit matrix\\
 & 7 & 0 &  Lemma \ref{fp_3}, no fixed part of a column orbit matrix\\
 & 10 & 0 &  Lemma \ref{fp_3}, by exhaustive generation\\
 & 13 & 0 &  Lemma \ref{fp_3}, by exhaustive generation\\
 & 16 & 0 &  Lemma \ref{fp_3}, indexing of the nonfixed part not possible\\
 & 19 & 0 &  Lemma \ref{fp_3}, indexing of the nonfixed part not possible\\
5 & 2 &0 & Lemma \ref{aut5}, no prototypes  \\
 & 7 &0 & Lemma \ref{aut5}, no fixed part of a column orbit matrix\\
 & 12 &0 & Lemma \ref{aut5}, no fixed part of a column orbit matrix\\
 & 17 &0 & Lemma \ref{aut5}, indexing of the nonfixed part not possible\\
7 & 1 & 0 & Lemma \ref{aut7}, no prototypes \\
 & 8 & 0 & Lemma \ref{aut7}, no fixed part of a column orbit matrix\\
 & 15 & 0 & Lemma \ref{aut7}, indexing of the nonfixed part not possible\\
11 & 0 & 0& Lemma \ref{aut11}, no column orbit matrices\\
 & 11 & 0& Lemma \ref{aut11}, no fixed part of a column orbit matrix\\
13 & 9 & 0& Lemma \ref{aut13}, no prototypes \\
17 & 5 & 0& Lemma \ref{aut13}, no prototypes\\
19 & 3 &0 & Lemma \ref{aut13}, no prototypes\\
\end{tabular}}\] 
\end{table}
\end{center} 
 
The results of our work, which are given in Table \ref{summ}, can be summarized in Theorem \ref{final} after the elimination of isomorphic copies.

\begin{theorem} \label{final}
There are at least $472$ nonisomorphic DSRGs$(22,9,6,3,4)$ with a nontrivial automorphism, with the full automorphism groups as given in Table \ref{dsrg_22}. Moreover, if there is a DSRG$(22,9,6,3,4)$ with a nontrivial automorphism that is not in Table \ref{dsrg_22}, then the order of its full automorphism group is a power of two, and an automorphism of order two fixes $0, 10, 12, 14$ or $16$ vertices.
 \begin{center} \footnotesize
	\begin{table}[htpb!] 
		\caption{\label{dsrg_22} DSRGs(22,9,6,3,4) with a nontrivial automorphism}
		\[ { \begin{tabular}{c||c|c|c|c|c|c}
				the full automorphism group &$Z_{2}$&$Z_3$&$Z_4$&$Z_2\times Z_2$&$S_3$&$Z_4\times Z_2$\\
				\hline
				\hline
				$  \#DSRGs $&320 & 120& 4&10&14&4\\

		\end{tabular}}\] 
	\end{table}
\end{center} 

\end{theorem}

The adjacency matrices of all constructed DSRGs$(22,9,6,3,4)$  are available online at  
\begin{verbatim}
http://www.math.uniri.hr/~mmaksimovic/DSRG(22,9,6,3,4).pdf 
\end{verbatim}

\begin{center}{\bf Acknowledgement}\end{center}
This work has been supported by the University of Rijeka project uniri-iskusni-prirod-23-62.  The authors would like to thank Dean Crnkovi\'c for drawing their attention to the reference \cite{csz}.


\begin{thebibliography}{30}

\bibitem{beh}
M. Behbahani, C. Lam,
{\it Strongly regular graphs with nontrivial automorphisms}, Discrete Math. {\bf 311} (2011), 132--144.

\bibitem{BCN}
A. E. Brouwer, A. M. Cohen, A. Neumaier, 
{\it Distance-regular graphs}, Springer, 1989.

\bibitem{bro1} A. E. Brouwer, S. A. Hobart,  {\it Parameters of directed strongly regular graphs}, https://homepages.cwi.nl/~aeb/math/dsrg/dsrg.html. (accessed December 10, 2024)  


\bibitem{bro2} A. E. Brouwer, H. Van Maldeghem, {\it Strongly regular graphs}, Cambridge Univ. Press, 2022.


\bibitem{cm}
D. Crnkovi\'c, M. Maksimovi\'c, {\it Construction of strongly regular graphs having an automorphism group of composite order},  Contrib. Discrete Math. {\bf 15} (2020), 22--41.


\bibitem{c-r-metrika}
D. Crnkovi\'c, S. Rukavina, 
{\it Construction of block designs admitting an abelian automorphism group}, Metrika {\bf 62} (2005), 175--183.

\bibitem{csz}
D. Crnkovi\'c, A. \v Svob, T. Zrinski, {\it
Construction of directed strongly regular graphs via
their orbit matrices and genetic algorithm}, https://arxiv.org/abs/2412.14787

\bibitem{duval}A. M. Duval, 
{\it A directed graph version of strongly regular graphs},
J. Combin. Theory Ser. A {\bf 47} (1988), no. 1, 71–100.

\bibitem{janko}
Z. Janko, {\it Coset enumeration in groups and constructions of symmetric designs}, Combi-
natorics ’90 (Gaeta, 1990), Ann. Discrete Math. {\bf52} (1992), 275–277.

\bibitem{GAP}
The GAP Group, {\it GAP -- Groups, Algorithms, and Programming, version 4.8.10}, 2018, https://www.gap-system.org.




\end{thebibliography}
\end{document}